\documentclass{article}
\usepackage[utf8]{inputenc}
\usepackage[square,sort,comma,numbers]{natbib}
\usepackage{graphicx}
\usepackage{amsmath}
\usepackage{amsthm}
\usepackage{amsfonts}
\usepackage{amssymb}
\DeclareMathOperator{\spn}{span}
\newcommand{\ol}{\overline} 
\usepackage{tikz}
\usepackage[a4paper]{geometry}
\usepackage{capt-of}
\usepackage{float}
\newcommand{\vertiii}[1]{{\left\vert\kern-0.25ex\left\vert\kern-0.25ex\left\vert #1 
    \right\vert\kern-0.25ex\right\vert\kern-0.25ex\right\vert}}
 \title{Rigorous derivation of the generalized Reynolds equation from the Boltzmann equation}
 \author{Andrei Ichim}
 \date{}
\newtheorem{theorem}{Theorem}
\newtheorem{lemma}{Lemma}

\begin{document}
\maketitle
\begin{abstract}
We study the stationary Boltzmann equation in a thin slab for a rarefied gas for which the molecular mean free path is comparable to the film thickness. We prove that there exists a solution which converges, in the hydrodynamic limit, to a density Maxwellian -- this density is obtained as the solution to the generalized Reynolds equation. The convergence is proved using a truncated Hilbert expansion by carefully estimating the remainder. 
\end{abstract}
\section{Introduction}
\subsection{The problem and its motivation}
The Reynolds equation describes the pressure distribution in a thin layer of lubricant film between two surfaces in relative motion. More precisely, let the surfaces be located at $z=0$ and $z=\varepsilon H(x,y)$ with $\varepsilon\ll 1$. Assume that the top surface is fixed while the bottom one moves with constant velocity $U$ in the $x$ direction. The Reynolds equation can be written as
\begin{equation}\label{reynolds0}
    \partial_x \left(H^3\partial_x p\right)+\partial_y \left(H^3\partial_y p\right)=6U\partial_x H. 
\end{equation}
The equation \eqref{reynolds0} was derived in a heuristic way by Reynolds (1886) from the Navier-Stokes equations. He noticed that, by averaging the mass conservation (from $0$ to $H(x,y)$) and using the momentum balance equation to evaluate the quantities appearing as integrands we can eliminate the dependence on the velocity and also on the spatial variable $z$ to derive \eqref{reynolds0}. Since the Reynolds equation has had numerous practical applications (see for instance \cite{cercignani2005}) -- it was naturally addressed by physicists and engineers the question whether Reynolds's argument can be extended to rarefied gases. Since in the context of gas lubrication the molecular mean free path
is no longer negligible when compared with the to the macroscopic distance between the two surfaces the continuum hypothesis fails and the kinetic theory needs to be employed. In \cite{fk88} the authors start from the linearized BGK equation and obtain heuristically a modified Reynolds equation  which accounts for the kinetic effects. \\
The aim of this paper is to present a rigorous derivation of the generalized Reynolds equation starting from the fully nonlinear (stationary) Boltzmann equation. To achieve this we are going to use (as in \cite{caflisch} or \cite{esposito1994}) a truncated expansion in terms of the Knudsen number $Kn\sim\varepsilon$ whose leading term is a density Maxwellian. This density $\rho$ will be the solution to the generalized Reynolds equation, which will be obtained as an averaged mass conservation of the second order term of the Hilbert expansion. The main result in Section 2 is the proof of the existence and critically the positivity of $\rho$.\\
The remainder of the paper will be devoted to establishing a bound on the remainder. We are going to use the techniques developed by R. Esposito, Y. Guo et al. (\cite{esposito2013}, \cite{esposito2018}) in order to obtain the delicate $L^2$ estimates on the hydrodynamic part of the solution. The only real difference is related to the fact that the size of the domain itself depends on $\varepsilon$, so close attention is required in order to see how various constants depending on the domain change with respect to $\varepsilon$.  

\subsection{Preliminary definitions and results}
\indent We consider $\omega\subset\mathbb{R}^2$ be a $C^\infty$ bounded open set whose boundary is described by $\partial\omega=\{(x,y)\in\mathbb{R}^2\,|\,x\in [0,1],y=y_0(x)\} $ with $y_0$ satisfying
\begin{equation}\label{eqy0}
    \sup_{x\in [0,1]}|y_0''(x)|<\infty. 
\end{equation}
Let $H>0$ be fixed. For all small $\varepsilon>0$ let $D_\varepsilon=\omega\times [0,\varepsilon H]$ and we will call $D=[0,H]$ the rescaled domain. The three part boundary of $D_\varepsilon$ can be written as
\[
\partial D_\varepsilon=\underbrace{\omega_b\cup \omega_t}_{:=\omega'}\cup \gamma_l^\varepsilon,
\]
where $\omega_b=\omega\times\{z=0\}$, $\omega_t=\omega\times\{z=H\}$ and $\gamma_l^\varepsilon=\partial\omega\times (0,\varepsilon H)$ (we will call $\gamma_l=\partial\omega\times (0,H)$) -- see figure below.\\
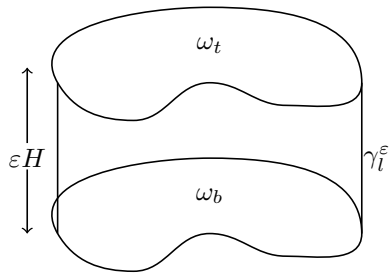
\begin{figure}[H]
	\begin{tikzpicture}
	\draw[semithick] (0,0) to [out=300,in=180] (1,-.5) to [out=0,in=180] (2,0) to [out=0,in=180] (3,-.3) to [out=0,in=270] (4,0) to [out=90,in=0] (2,1) to [out=180,in=120] (0,0);
	\draw[semithick] (0,2) to [out=300,in=180] (1,1.5) to [out=0,in=180] (2,2) to [out=0,in=180] (3,1.7) to [out=0,in=270] (4,2) to [out=90,in=0] (2,3) to [out=180,in=120] (0,2);
	\draw[-,semithick] (0,0) -- (0,2);
	\draw[-,semithick] (4,0) -- (4,2);
	\draw (2,.5) node {$\omega_b$};
	\draw (2,2.5) node {$\omega_t$};
	\draw (4.2,1) node {$\gamma^\varepsilon_l$};
	\draw[->,semithick] (-0.4,1.2) -- (-0.4,2.2);
	\draw[->,semithick] (-0.4,.85) -- (-0.4,0);
	\draw (-.4,1) node {$\varepsilon H$};
	\draw (-5,.5) node {};
        .. (3, -1) .. controls (4, 0) and (2, 1)
   .. (0, 0);    
	\end{tikzpicture}
\caption{The domain $D_\varepsilon$.}
\end{figure}
We consider the stationary Boltzmann equation in the domain $D_\varepsilon$ and assume that the Knudsen number is of order $\varepsilon$
\begin{equation}\label{bltze}
    v\cdot\nabla_\bold{x} f_\varepsilon=\frac{1}{k_0\varepsilon}Q(f_\varepsilon,f_\varepsilon),\quad\text{in } D_\varepsilon,
\end{equation}
where $f_\varepsilon$ is the distribution density, $k_0$ is the rescaled Knudsen number and the Boltzmann collision operator $Q$ corresponding to the hard spheres cross section is defined as follows
\begin{equation*}
    Q(f,f)=\int_{\mathbb{R}^3}\int_{\mathbb{S}^2}\left(f(v_*')f(v')-f(v_*)f(v)\right)|(v-v_*)\cdot\sigma|dv_* d\sigma,
\end{equation*}
with $ v' = v-\left((v - v_*) \cdot \sigma\right)\sigma,\,v_*' = v+\left((v - v_*) \cdot \sigma\right)\sigma$. Boltzmann’s collision operator has the fundamental property of conserving mass, momentum and energy
\begin{equation}\label{masscons}
    \int_{\mathbb{R}^3}Q(f,f) \begin{pmatrix} 1\\v\\|v|^2\end{pmatrix}dv=0.
\end{equation}
To simplify the calculations, throughout this material we will only be using a normalized Maxwellian $M$ defined by
$$M=\frac{1}{(2\pi)^{3/2}}e^{-\frac{|v|^2}{2}}.$$
We introduce the linearized collision operator
\begin{equation}\label{linearized}
    Lf=-\frac{2M^{-1}}{k_0}Q(M,Mf),
\end{equation}
as well as the nonlinear collision operator
\begin{equation*}
    \Gamma(f,g)=\frac{2M^{-1}}{k_0}Q(Mf,Mg).
\end{equation*} 
The operator $L:L^2(Mdv)\to L^2(Mdv)$ is self-adjoint and has a five dimensional null space
\begin{equation}\label{eqLker}
    \ker L=\spn\{1,v,|v|^2\}.
\end{equation}
We denote the orthogonal projection of $f$ onto $\ker L$ as
\begin{equation}\label{projdef}
    Pf=a+v\cdot b+\frac{|v|^2-3}{2}c,
\end{equation}
while the projection on the orthogonal complement of $\ker L$ we will call 
\begin{equation}\label{kindef}
f^\perp=f-Pf.    
\end{equation}

It is well known (for instance from \cite{cercignani1994}) that $$L=\nu I+K,$$ where
the collision frequency $\nu=\nu(v)$ satisfies the following property
\begin{equation}\label{eqnu}
\nu_m(1+|v|)\leq \nu(v)\leq\nu_M( 1+|v|)    
\end{equation}
for some $\nu_m,\nu_M>0$. Moreover, the operator $K:L^2(Mdv)\to L^2(Mdv)$ is compact, which readily implies the following:
\begin{align}
    \int_{\mathbb{R}^3}Lg h Mdv\lesssim \|\nu^{1/2} {g}\|_{L^2(Mdv)}\|h\|_{L^2(Mdv)}.\label{eqLcont}
\end{align}
The operator $L$ is symmetric in $L^2(Mdv)$ i.e.
\begin{equation}\label{eqLsim}
    \int_{\mathbb{R}^3}Lg h Mdv=\int_{\mathbb{R}^3}Lh g Mdv
\end{equation}
The operator $L$ also satisfies the following spectral inequality:
\begin{equation}\label{coer}
    \|\nu^{1/2} {g^{\perp}}\|_{L^2(Mdv)}^2\lesssim\int_{\mathbb{R}^3}gLg Mdv.
\end{equation}
\subsection{Boundary conditions and notations}
 The interaction of gas with the boundaries $\omega_b\cup \omega_t$ is modelled by diffuse reflection boundary condition, namely
\begin{equation}\label{bci1}
    f{_\varepsilon(\bold{x},v)}_{v\cdot n_\bold{x}<0}=\sqrt{2\pi}M\int_{w\cdot n_\bold{x}>0}f_\varepsilon(\bold{x},w) w\cdot n_\bold{x} dw:=\beta_{f_{\varepsilon}}(\bold{x}),
\end{equation}
where for any subset $\gamma^s\subset \partial D_\varepsilon$ we call
$$\gamma^s_{\pm}=\{(\bold{x},v)\in \gamma^s\times\mathbb{R}^3\,|\,v\cdot n_\bold{x}\gtrless 0\},$$
with $n_\bold{x}$ the outward normal at $\bold{x}\in\gamma_s$. The boundary condition \eqref{bci1} ensures the zero net mass flow at the
top and bottom boundaries:
$$ \int_{\mathbb{R}^3}f_\varepsilon(\bold{x},v) v\cdot n_\bold{x} dv=0\quad\text{for all }\bold{x}\in\omega'. $$
On $\gamma_l^\varepsilon$, we are going to assume the following condition: 
\begin{equation}\label{bci2}
f{_\varepsilon}_{\big|(\gamma_l^\varepsilon)_{-}}=\rho_0 M+\varepsilon \ol{g^1},    
\end{equation}
where $\rho_0:\partial\omega\to\mathbb{R}$ is given, with $\rho_0\in C^\infty(\partial\omega)$ and $\rho_0>0$ on $\partial\omega$. The function $\ol{g^1}$ will be later specified.
\vspace{0.5cm}\\
\emph{Notations.} We denote $\bold{x}=(x,y,z)$ the points in the physical space and $v=(v_x,v_y,v_z)$ the microscopic velocities. We will always use the subscript $\varepsilon$ to indicate functions which depend on $\varepsilon$. In order not to complicate the notations, we use $\|\cdot\|$ to denote both the $L^2\left(D_\varepsilon;L^2(Mdv)\right)$ and the $L^2\left(0,H;L^2(Mdv)\right)$ norms. Likewise, we denote $\|\cdot\|_2$ the both $L^2(D_\varepsilon)$ and the $L^2(D)$ norms. We also use $\| f\|_{\nu}=\| \nu^{1/2}f \|$. We call $(\cdot,\cdot)$ the scalar product on $L^2(Mdv)$. We define $d\gamma=|v\cdot n_\bold{x}|dS(\bold{x})$ where $dS(\bold{x})$ is the surface measure and define the $L^2$ norm $|f|_{\gamma^s}^2=\int_{\gamma^s}|f|^2 d\gamma$. Moreover, $|f|_{\gamma^s,\pm}=|f\bold{1}_{\gamma^s_{\pm}}|$. Finally, we use the notation $X\lesssim Y$ to say that $X\leq CY$ for some constant independent of $X$ and $Y$ and also \emph{independent of $\varepsilon$}.

\section{The Hilbert expansion and the generalized Reynolds equation}
\subsection{Study of $\mathcal{L}_\rho$}
 By scaling $\hat{z}=\varepsilon^{-1} z$ we can rewrite the equation \eqref{bltze} in the domain $D$ -- dropping the hats -- as
\begin{equation}\label{bltz}
     v_x\partial_x f_\varepsilon+ v_y\partial_y f_\varepsilon+\frac{1}{\varepsilon}v_z\partial_z f_\varepsilon=\frac{1}{k_0\varepsilon}Q(f_\varepsilon,f_\varepsilon).
\end{equation}

Consider the following formal expansion of $f$
\begin{equation}\label{expansion}
    f=\rho M+\varepsilon f^1+\varepsilon^2 f^2+\cdots,
\end{equation}
where $\rho=\rho(x,y)>0$ is a function to be determined.\\
The boundary conditions are expected to become
\begin{align}
    {f^{m}(\bold{x},v)}_{\big|\omega'_{-}}=\sqrt{2\pi}M\int_{w\cdot n_{\bold{x}}>0}f^m(\bold{x},w) w\cdot n_{\bold{x}} dw\quad\forall\,m\geq 1,\,\bold{x}\in\omega',\label{bc1}\\
    \rho=\rho_0\quad \text{on }\partial\omega,\label{bc2}\\
    {f^{1}}_{\big|(\gamma_l)_{-}}=M\ol{g^1}\quad\text{on }\gamma_l.\label{bc3}
\end{align}
By plugging \eqref{expansion} into \eqref{bltz} and identifying the powers of $\varepsilon$ we find that:
\begin{align}
    &\varepsilon^0: M(v_x\partial_x \rho+ v_y\partial_y \rho)+v_z\partial_z f^1=\frac{2\rho}{k_0}Q(M,f^1)\label{ord1}\\
    &\varepsilon^1: v_x\partial_x f^1+ v_y\partial_y f^1+v_z\partial_z f^2=\frac{2\rho}{k_0}Q(M,f^2)+\frac{1}{k_0}Q(f^1,f^1)\label{ord2}
\end{align}
By writing $f^1=Mg^1$ and recalling the definition of the operator $L$ in \eqref{linearized} we can rewrite the equation \eqref{ord1} as
\begin{equation}\label{eqord1}
  v_z\partial_z g^1+\rho L g^1= -\left(v_x\partial_x \rho+ v_y\partial_y \rho\right).
\end{equation}
It is natural to introduce the operator
\begin{equation}\label{Lrhodef}
    \mathcal{L}_\rho=v_z\partial_z +\rho L
\end{equation}
in order to study problem \eqref{eqord1}. Since $\mathcal{L}_\rho$ acts on the spatial variable $z$ alone, it is convenient to look at it for fixed $(x,y)\in\omega$ as being defined on (a subspace of) the space $L^2(0,H;L^2(Mdv))$. 
From \eqref{bc1} it follows that the boundary conditions for $g^1$ become:
\begin{equation}\label{bcg1}
    \begin{aligned}
    g^{1}(0)_{\big|v_z>0}&=\sqrt{2\pi}\int_{v_z<0}g^1(0,v)|v_z|M dv:=\beta_{g^1}(0),\\
    g^{1}(H)_{\big|v_z<0}&=\sqrt{2\pi}\int_{v_z>0}g^1(H,v)v_zM dv:=\beta_{g^1}(H).
    \end{aligned}
\end{equation}
The remainder of this section will be devoted to the careful study of the problem
\begin{equation}\label{eqLrho}
    \mathcal{L}_\rho g=h
\end{equation}
with the boundary conditions \eqref{bcg1}, that we explicitly write as
\begin{equation}\label{fullinear}
\left\{
\begin{aligned}
      &v_z\partial_z g +\rho Lg=h,\\
      g(0)_{\big|v_z>0}&=\beta_g(0),\,g(H)_{\big|v_z<0}=\beta_{g}(H).
\end{aligned}
  \right.
\end{equation}
 We will accomplish this in four steps, following the general lines set forth in \cite{esposito2018}. The proof will be significantly shorter since we will be making all the simplifications related to our one-dimensional setup. We also point out that we will be following the dependence of various constants with respect to $\rho$.\\
 \vspace{0.2cm}\\
\textbf{Step 1.} Start with the problem
\begin{equation}\label{eqstep1}
    \left\{
    \begin{aligned}
    &v_z\partial_z g+\rho\nu g=h,\\
    g(0)_{\big|v_z>0}&=g_0,\,g(H)_{\big|v_z<0}=g_H,
    \end{aligned}
    \right.
\end{equation}
where $g_0,g_H\in\mathbb{R}$ are fixed. Then we can explicitly write the solution to problem \eqref{eqstep1} as 
\begin{equation}\label{step1-2}
\begin{aligned}
    g(z,v)&=g_0e^{-\frac{\rho\nu}{v_z}z}+\int_0^z e^{\frac{\rho\nu}{v_z}(z'-z)}\frac{h(z')}{v_z}dz'\quad\text{for }v_z>0,\\
    g(z,v)&=g_He^{\frac{\rho\nu}{|v_z|}(z-H)}+\int_z^H e^{\frac{\rho\nu}{|v_z|}(z-z')}\frac{h(z')}{|v_z|}dz'\quad\text{for }v_z<0.
    \end{aligned}
\end{equation}
The uniqueness of the solution follows from the above representation. The plan is to show 
\begin{equation}\label{step1-3}
    h\in L^2(0,H;L^2(Mdv))\implies g\in L^2(0,H;L^2(Mdv)).
\end{equation}
 Let $h_\delta=h\bold{1}_{[\delta,1)}$ for $0<\delta\ll 1$. We will show \eqref{step1-3} with $h_\delta$ in place of $h$. Inspecting closely the solution \eqref{step1-2} we see that the only non-trivial part is to show that
\begin{equation*}
    \int_\delta^H\int_{\mathbb{R}^3}\frac{1}{v_z^2} e^{-\frac{2\rho\nu}{|v_z|}z-{v_z^2}}dvdz<\infty.
\end{equation*}
Noting that the convergence of the above integral is only problematic close to $v_z=0$ and using \eqref{eqnu} we obtain
\begin{equation*}
\begin{aligned}
      \int_\delta^H\int_{\mathbb{R}^3}\frac{1}{v_z^2} e^{-\frac{2\rho\nu}{|v_z|}z-v_z^2}dvdz\leq &c_1+c_2\int_\delta^H\int_0^1 \frac{1}{t^2} e^{-\frac{c_3}{t}z-t^2}dtdz\\
      \leq &c_1+c_4 \left(T_{-1}(c_3\delta)-T_{-1}(c_3 H)\right)<\infty,
\end{aligned}
\end{equation*}
where $c_1,c_2,c_3.c_4>0$ and $T_{-1}(z)=\int_0^\infty \frac{1}{t} e^{-\frac{z}{t}-t^2}dt$ is the Abramowitz function (see, for instance \cite{abramowitz}).\\
Clearly $h_\delta\to h$ in $L^2(0,H;L^2(Mdv))$ by Lebesgue's dominated convergence theorem. To conclude the proof of \eqref{step1-3}, we will be using the estimates that we prove at Step 3.\\
\vspace{.2cm}\\
\textbf{Step 2.} At this step, we prove an existence result in the case of diffusive reflection boundary conditions, namely for the problem
\begin{equation}\label{eqstep2}
    \left\{
    \begin{aligned}
    &v_z\partial_z g+\rho\nu g=h,\\
    g(0)_{\big|v_z>0}&=\beta_g(0),\,g(H)_{\big|v_z<0}=\beta_{g}(H).
    \end{aligned}
    \right.
\end{equation}
The proof follows very closely that in \cite{esposito2018} and we will only sketch it. For $0<\vartheta<1$ we construct the sequence $(g_l^\vartheta)$ as
\begin{equation*}
    \left\{
    \begin{aligned}
    &v_z\partial_z g_{l+1}^\vartheta+\rho\nu g_{l+1}^\vartheta=h,\\
    g_{l+1}^\vartheta(0)_{\big|v_z>0}&=\vartheta\beta_{g_{l}}^\vartheta(0),\,g_{l+1}^\vartheta(H)_{\big|v_z<0}=\vartheta\beta_{g_{l}}^\vartheta(H),
    \end{aligned}
    \right.
\end{equation*}
for $l\geq 0$, with $g_0^\vartheta=0$. The construction is possible owing to the result from the previous step. The proof uses energy estimates (again, proven in Step 3) and is done in two steps as follows:
\begin{itemize}
    \item Show that for fixed $\vartheta<1$, $g_l^\vartheta$ is Cauchy in some norm and then $g_l^\vartheta\to g^\vartheta$.
    \vspace{-.3cm}\\
    \item Prove $g^\vartheta\to g$ in $L^2(0,H;L^2(Mdv))$, with $g$ the desired solution to \eqref{eqstep2}.
\end{itemize}
The uniqueness of $g$ will follow immediately from energy estimates. We call this solution
\begin{equation}\label{step2solop}
    g=\mathcal{S}h.
\end{equation}
\vspace{.2cm}\\
\textbf{Step 3.} Using Green's formula in \eqref{fullinear} we find
\begin{equation}\label{step3-1}
    \left(v_z,\frac{g^2(H)}{2}\right)-\left(v_z,\frac{g^2(0)}{2}\right)+\rho\int_0^H (Lg,g)dz=\int_0^H (g,h)dz.
\end{equation}
An elementary calculation leads to
\begin{equation}\label{step3-2}
    \left(v_z,\frac{g^2(H)}{2}\right)-\left(v_z,\frac{g^2(0)}{2}\right)=|g(0)-\beta_g(0)|_+^2+|g(H)-\beta_g(H)|_+^2:=A^2(g)>0.
\end{equation}
Using \eqref{coer} and \eqref{step3-2} in \eqref{step3-1} we find that
\begin{equation}\label{step3estgreen}
    A^2(g)+\rho \|g^\perp\|^2_\nu\lesssim \lambda \|g\|^2+C_\lambda\|h\|^2
\end{equation}
for all $\lambda>0$. Clearly, for $\lambda$ small enough
\begin{equation}\label{step3estkin}
    \|g^\perp\|_\nu\lesssim\frac{1}{\rho}\|h\|.
\end{equation}
Let us now turn our attention to the case $h=0$. From \eqref{step3estkin} it follows that $g^\perp=0$, and so from \eqref{kindef} and \eqref{projdef} we can write
\begin{equation}\label{step3fluiddec}
    g(z,v)=a(z)+b^1(z)v_x+b^2(z)v_y+b^3(z)v_z+c(z)\frac{|v|^2-3}{2}.
\end{equation}
Plugging \eqref{step3fluiddec} into \eqref{fullinear} we readily find
\begin{equation}\label{step3-2}
    a'(z)=\left(b^i\right)'(z)=c'(z)=0\quad\text{for }z\in (0,H),\,i\in\{1,2,3\}.
\end{equation}
Using the boundary conditions in \eqref{fullinear} we get that $g(0)$ is independent of $v$ at least for $v_z>0$, and so clearly 
\begin{equation}\label{step3-3}
    b^i(0)=c(0)=0\quad\text{for }i\in\{1,2,3\}.
\end{equation}
From \eqref{step3fluiddec}, \eqref{step3-2} and \eqref{step3-3} we obtain
\begin{equation*}
    g(z,v)=a(=a(x,y)).
\end{equation*}
Clearly, every such function solves \eqref{fullinear}, and so we have completely found the kernel of $\mathcal{L}_\rho$ as
\begin{equation*}
    \ker\mathcal{L}_\rho=\spn\{a(x,y)\}.
\end{equation*}
Therefore any solution $g$ to problem \eqref{fullinear} can be written as
\begin{equation*}
    g(z,v)=a(x,y)+\underbrace{b^1(z)v_x+b^2(z)v_y+b^3(z)v_z+c(z)\frac{|v|^2-3}{2}+g^\perp(z,v)}_{:=\mathcal{L_\rho}^{-1}h}.
\end{equation*}
The term $\mathcal{L_\rho}^{-1}h$ is uniquely (and so properly) defined and we will call this \emph{the solution to the problem \eqref{fullinear}} (the actual existence of this solution will be proven at Step 4).\\
Let us know establish estimates for the fluid part of the solution $g$ to problem \eqref{fullinear}. Note that this will be significantly easier to do when compared to the three dimensional case in \cite{esposito2018} (or the one we deal with in the last section of this paper).\\
By integrating in \eqref{fullinear} from to $0$ to $z$ we find:
\begin{equation}\label{step3-3}
    v_zg(z,v)-v_zg(0,v)+\rho\int_0^z Lg(z',v)dz'=\int_0^z h(z',v)dz'. 
\end{equation}
Recall from earlier that, since $g$ is the solution to \eqref{fullinear}
\begin{equation}\label{step3-4}
    g(z,v)=b^1(z)v_x+b^2(z)v_y+b^3(z)v_z+c(z)\frac{|v|^2-3}{2}+g^\perp(z,v).
\end{equation}
We show how to obtain estimates on each individual component of $Pg$ as follows:\\
\vspace{.1cm}\\
\underline{Estimates on $b^1$ and $b^2$:} Take the scalar product in \eqref{step3-3} with respect to $\phi^1=v_xv_z$, and use \eqref{step3-4} to find:
\begin{equation}\label{step3-5}
    \left(g(z,v),v_xv_z \right)=\frac{1}{4}b^1(z)+\left(g^\perp(z,v),v_xv_z\right),
\end{equation}
since the other terms contributing to $Pg$ vanish due to oddness in $v$ and an elementary computation leads to $\left(v_xv_z,v_xv_z\right)=\frac{1}{4}$. Note first that:
\begin{equation}\label{step3-6}
    \left|\left(g^\perp(z,v),v_xv_z\right)\right|\leq \|g^\perp\|.
\end{equation}
Concerning the boundary term we can write:
\begin{equation}\label{step3-7}
   \left|\left(v_zg(0),v_xv_z \right)\right|=\left|\left(v_zg(0)-v_z\beta_g(0),v_xv_z \right)\right|\lesssim \left| g(0)-\beta_g(0) \right|_+\leq A(g),
\end{equation}
using the fact that the contribution of $\beta_g(0)$ vanishes due to oddness in $v$.\\
Using \eqref{eqLcont}, we can bound the last two terms in \eqref{step3-3} as follows:
\begin{align}
    \rho\left|\int_0^z \left(Lg(z',v),v_xv_z\right) dz'\right|&\lesssim \rho \| g^\perp\|_\nu,\label{step3-8}\\
    \left|\int_0^z \left(h(z',v),v_xv_z\right)dz'\right|&\lesssim\|h\|.\label{step3-9}
\end{align}
The term $b^2$ is treated in a very similar fashion -- by choosing $\phi^2=v_yv_z$ as a test function instead of $\phi^1$ -- and the estimates obtained are exactly the same. In conclusion, by combining \eqref{step3-4}--\eqref{step3-9} we obtain
\begin{equation}\label{estb1}
    \left|b^i(z)\right|\lesssim A(g)+\|g^\perp\|+\rho \| g^\perp\|_\nu+\|h\|\quad\text{for }i\in\{1,2\}.
\end{equation}
\vspace{.1cm}\\
\underline{Estimate on $b^3$:}  By choosing as test function in \eqref{step3-3} $\phi^3=1$ and using \eqref{step3-4} we obtain:
\begin{equation}\label{step3-10}
    \left(g(z,v),1 \right)=\frac{1}{2}b^3(z),
\end{equation}
as the other terms contributing to $Pg$ vanish due to oddness in $v$ and $\left(v_z,v_z\right)=\frac{1}{2}$. Moreover $\left(g^\perp(z,v),1\right)=0$ since $g^\perp\in (\ker L)^\perp$.\\
Furthermore, we have
\begin{equation}\label{step3-11}
   \left|\left(v_zg(0),1 \right)\right|=\left|\left(v_zg(0)-v_z\beta_g(0),1 \right)\right|\lesssim \left| g(0)-\beta_g(0) \right|_+\leq A(g),
\end{equation}
as the fact that the contribution of $\beta_g(0)$ vanishes due to oddness in $v$.\\
From \eqref{eqLsim} and \eqref{eqLker} we find
\begin{equation}\label{step3-12}
    \int_0^z \left(Lg(z',v),1\right) dz'=0.
\end{equation}
Lastly, we have
\begin{equation}\label{step3-13}
    \left|\int_0^z \left(h(z',v),1\right)dz'\right|\lesssim\|h\|.
\end{equation}
By taking into account \eqref{step3-4} and \eqref{step3-10}--\eqref{step3-13} we derive:
\begin{equation}\label{estb3}
    \left|b^3(z)\right|\lesssim A(g)+\|h\|.
\end{equation}
\vspace{.1cm}\\
\underline{Estimate on $c$:} Take as test function in \eqref{step3-3} $\phi^4=v_z(|v|^2-5)$, and use \eqref{step3-4} to find:
\begin{equation}\label{step3-14}
    \left(g(z,v),v_z(|v|^2-5) \right)= 10\pi b^1(z)+\left(g^\perp(z,v),v_xv_z\right),
\end{equation}
since, once again, the other terms contributing to $Pg$ vanish due to oddness in $v$ and the following Gaussian integral can be computed $\left(v_z(|v|^2-5),v_z\frac{|v|^2-3}{2}\right)=10\pi$. Clearly we have
\begin{equation}\label{step3-15}
    \left|\left(g^\perp(z,v),v_z(|v|^2-5)\right)\right|\leq \|g^\perp\|.
\end{equation}
Moreover, the boundary term can be written as:
\begin{equation}\label{step3-16}
   \left|\left(v_zg(0),v_z(|v|^2-5) \right)\right|=\left|\left(v_zg(0)-v_z\beta_g(0),v_z(|v|^2-5) \right)\right|\lesssim \left| g(0)-\beta_g(0) \right|_+\leq A(g),
\end{equation}
since we have the critical cancellation $\left(v_z,v_z(|v|^2-5) \right)=0 $.\\
Just as in the estimate for $b^1$ we easily find:
\begin{align}
    \rho\left|\int_0^z \left(Lg(z',v),v_z(|v|^2-5)\right) dz'\right|&\lesssim \rho \| g^\perp\|_\nu,\label{step3-17}\\
    \left|\int_0^z \left(h(z',v),v_z(|v|^2-5)\right)dz'\right|&\lesssim\|h\|.\label{step3-18}
\end{align}
In conclusion, by using \eqref{step3-4} and \eqref{step3-14}--\eqref{step3-18} we obtain
\begin{equation}\label{estc}
    \left|c(z)\right|\lesssim A(g)+\|g^\perp\|+\rho \| g^\perp\|_\nu+\|h\|.
\end{equation}
By combining \eqref{estb1}, \eqref{estb3} and \eqref{estc} to find
\begin{equation}\label{step3estpg}
    \|Pg\|\lesssim A(g)+\|g^\perp\|+\rho \| g^\perp\|_\nu+\|h\|.
\end{equation}
Note that, from \eqref{eqnu} we can easily get
\begin{equation}\label{step3-19}
    \|g^\perp\|\lesssim\| g^\perp\|_\nu.
\end{equation}
Finally, we can use \eqref{step3estgreen}, \eqref{step3estkin}, \label{step3estpg} and \label{step3-19} to derive that, for $\lambda$ small enough
\begin{equation*}
    \|g\|\lesssim \left(1+\frac{1}{\rho}\right)\|h\|.
\end{equation*}
\vspace{.2cm}\\
\textbf{Step 4.} Let now $\Tilde{g}\in L^2(0,h;L^2(Mdv))$ and consider the problem
\begin{equation}\label{eqstep4}
    \left\{
    \begin{aligned}
    &v_z\partial_z g+\rho\nu g=h-\rho K\Tilde{g},\\
    g(0)_{\big|v_z>0}&=\beta_g(0),\,g(H)_{\big|v_z<0}=\beta_{g}(H),
    \end{aligned}
    \right.
\end{equation}
From Step 2. (see \eqref{step2solop}) this problem has a unique solution
\begin{equation*}
    g=\mathcal{S}h-\rho\mathcal{S} K\Tilde{g}.
\end{equation*}
The estimates in the previous step show, in particular that $\mathcal{S}$ is bounded in $L^2(0,H;L^2(Mdv))$. The compactness of $K$ implies that $\mathcal{S}K$ is compact in $L^2(0,H;L^2(Mdv))$. Still using the estimates in Step 3 it is elementary to show that
\begin{equation*}
    \ker \left(I+\rho \mathcal{S}K\right)=\{0\},
\end{equation*}
and so we can employ the Fredholm alternative to show that the problem
\begin{equation*}
    g=\mathcal{S}h-\rho\mathcal{S} Kg.
\end{equation*}
is solvable, which yields the desired solution to problem \eqref{fullinear}.\\
Let us collect the results we have obtained thus far in this following:
\begin{theorem}\label{thm1}
The problem \eqref{fullinear} has a unique solution 
\begin{equation*}
    g=\mathcal{L}_\rho^{-1}h,\quad\text{with } (g(z),1)=0\,\forall\,z\in [0,H].
\end{equation*}
Moreover, the operator $\mathcal{L}_\rho^{-1}$ is bounded in $L^2(0,H;L^2(Mdv))$ and
\begin{equation}\label{Lrhoest}
    \left\|\mathcal{L}_\rho^{-1}\right\|\lesssim 1+\frac{1}{\rho}.
\end{equation}
A general solution to problem \eqref{fullinear} can be written as
\begin{equation}
    g=a(x,y)+\mathcal{L}_\rho^{-1}h.
\end{equation}
\end{theorem}
\subsection{The generalized Reynolds equation}
Let us now turn our attention to the problem \eqref{eqord1} with boundary conditions \eqref{bcg1}. In view of Theorem \ref{thm1} the general solution can be written as:
\begin{equation*}
    g^1=a^1(x,y)-\mathcal{L}_\rho^{-1}\left(\partial_x \rho v_x+\partial_y \rho v_y \right).
\end{equation*}
Since $\mathcal{L}_\rho^{-1}$ is linear and acts on the spatial $z$ variable alone, we can further write:
\begin{equation}\label{solg1}
    g^1=a^1(x,y)-\partial_x \rho\mathcal{L}_\rho^{-1}v_x-\partial_y \rho\mathcal{L}_\rho^{-1}v_y.
\end{equation}
Taking the scalar product in $L^2(dv)$ with respect to $1$ in \eqref{ord2} and integrating from $0$ to $H$ we obtain, using \eqref{masscons}:
\begin{equation}\label{masscons2}
    \partial_x\int_0^H\int_{\mathbb{R}^3}v_xf^1dvdz+\partial_y\int_0^H\int_{\mathbb{R}^3}v_yf^1dvdz=\int_{\mathbb{R}^3}v_zf^2(0)dv-\int_{\mathbb{R}^3}v_zf^2(H)dv.
\end{equation}
From \eqref{bc1} applied to $m=2$ we deduce that $f^2$ solves the no-flux condition at the boundary $\omega'$, and so we have
\begin{equation}\label{nofluxf2}
    \int_{\mathbb{R}^3}v_zf^2(0)dv=\int_{\mathbb{R}^3}v_zf^2(H)dv=0.
\end{equation}
Recalling that $f^1=Mg^1$ we can plug \eqref{solg1} into \eqref{masscons2} and use \eqref{nofluxf2} to find:
\begin{equation}\label{reynoldsgen0}
 \partial_x\left( \partial_x\rho A_{xx}(\rho) \right)+\partial_x\left( \partial_y\rho A_{xy}(\rho) \right)+
    \partial_y\left( \partial_x\rho A_{yx}(\rho) \right)+\partial_y\left( \partial_y\rho A_{yy}(\rho) \right) =0,
\end{equation}
where
\begin{equation}\label{eqAA}
\left\{
    \begin{aligned}
    A_{xx}(\rho)&=    \int_0^H\left( v_x,\mathcal{L}_\rho^{-1}v_x \right)dz,\\
     A_{xy}(\rho)&=    \int_0^H\left( v_x,\mathcal{L}_\rho^{-1}v_y \right)dz,\\
      A_{yx}(\rho)&=    \int_0^H\left( v_y,\mathcal{L}_\rho^{-1}v_x \right)dz,\\
       A_{yy}(\rho)&=    \int_0^H\left( v_y,\mathcal{L}_\rho^{-1}v_y \right)dz.
\end{aligned}
\right.
\end{equation}

Let $g_x=\mathcal{L}_\rho^{-1}v_x, g_y=\mathcal{L}_\rho^{-1}v_y$, and so we have
\begin{equation*}
    \mathcal{L}_\rho\begin{pmatrix}g_x\\g_y\\0
    \end{pmatrix}=\begin{pmatrix}v_x\\v_y\\0
    \end{pmatrix}.
\end{equation*}
The next two results, largely inspired from \cite{desvilletes} will be very useful.
\begin{lemma}\label{lem1}
For all $g:\mathbb{R}^3\to\mathbb{R}^3$ and all isometries $R:\mathbb{R}^3\to\mathbb{R}^3$ which leave invariant the last coordinate i.e.
\[
R=\begin{pmatrix}\cos\theta & \sin\theta & 0\\
-\sin\theta & \cos\theta & 0\\
0 & 0 & 1
\end{pmatrix}
\]
for some $\theta\in [0,2\pi)$ we have the following
\[
\mathcal{L}_\rho\left(Rg(v)\right)=\left(\mathcal{L}_\rho g\right)(Rv),\quad\forall\,v\in\mathbb{R}^3.
\]
\end{lemma}
\begin{proof}
The proof follows very closely that of Lemma 1 in \cite{desvilletes}.
\end{proof}
\begin{lemma}
There exists a function $w:\mathbb{R}_+\times\mathbb{R}\times [0,H]\times\mathbb{R}$ such that 
\begin{equation}\label{eqww}
   \begin{aligned}
    g_x(v_x,v_y,v_z,z,\rho)&=w(v_x^2+v_y^2,v_z,z,\rho)v_x,\\
    g_y(v_x,v_y,v_z,z,\rho)&=w(v_x^2+v_y^2,v_z,z,\rho)v_y.
\end{aligned} 
\end{equation}

\end{lemma}
\begin{proof}
The proof uses the previous Lemma \ref{lem1} and follows very closely the lines of Lemma 3 in \cite{desvilletes}.  
\end{proof}

We can now prove the following result:
\begin{lemma}\label{lem3}[Properties of $A$ functions]
\begin{itemize}
    \item[(i)] $A_{xy}(\rho)=A_{yx}(\rho)=0$ and $A_{xx}(\rho)=A_{yy}(\rho):=A(\rho)$.\\
    \item[(ii)] $A(\rho)>0$ and $A\in C^\infty\left((0,\infty)\right)$, with
    \begin{equation}\label{Ainfty}
        A^{(n)}(\rho)=(-1)^n n! \int_0^H\big( \underbrace{\mathcal{L}_\rho^{-1}L\mathcal{L}_\rho^{-1}\dots L\mathcal{L}_\rho^{-1}}_{(n+1)\text{ times}}v_x,v_x\big)dz,\quad\forall n\geq 1.
    \end{equation}
\end{itemize}
\end{lemma}
\begin{proof}
The proof of $(i)$ follows immediately from \eqref{eqww} and a change to polar coordinates in \eqref{eqAA}.\\ To prove $(ii)$, note first that since $g_x=\mathcal{L}_\rho^{-1}v_x$ we can write
\begin{equation}\label{eqfx}
    v_z\partial_z g_x+\rho Lg_x=v_x.
\end{equation}
The Green's formula in \eqref{eqfx} tells us that
\begin{equation}\label{eqA}
    A(\rho)=\rho\int_0^H (Lg_x,g_x)dz+\int_{v_z<0}|v_z|\frac{g_x^2(0)}{2}dv+\int_{v_z>0}v_z\frac{g_x^2(H)}{2}dv.
\end{equation}
Since $\rho>0$, the positivity of $L$ clearly implies that $A(\rho)\geq 0$. If $A(\rho)=0$, that would imply all terms in \eqref{eqA} are $0$ and, in particular $Lg_x=0$, which means
\begin{equation}\label{eqfx2}
    g_x=\xi_1(z,\rho)v_x+\xi_2(z,\rho)v_y+\xi_3(z,\rho)v_z+\xi_4(z,\rho)|v|^2.
\end{equation}
Plugging \eqref{eqfx2} into \eqref{eqfx} we get a contradiction, and hence the initial assumption is false showing $A(\rho)>0$.\\
Finally we will show \eqref{Ainfty} for $n=1$, as the case $n>1$ is done easily by induction. Let $\rho,\epsilon>0$. We have
\begin{equation}\label{Lepsilon}
    \mathcal{L}_\epsilon^{-1}-\mathcal{L}_\rho^{-1}=\mathcal{L}_\epsilon^{-1}\left(\mathcal{L}_\rho-\mathcal{L}_\epsilon\right)\mathcal{L}_\rho^{-1}=(\rho-\epsilon)\mathcal{L}_\epsilon^{-1}L\mathcal{L}_\rho^{-1}
\end{equation}
from the definition of $\mathcal{L}_\rho$ in \eqref{Lrhodef}. From the estimate \eqref{Lrhoest} we know that $\mathcal{L}_\epsilon^{-1}$ is uniformly bounded in $\epsilon$ away from $0$. Hence from \eqref{Lepsilon} we deduce that
\begin{equation}\label{Lepsiloncont}
    \mathcal{L}_\epsilon^{-1}\to \mathcal{L}_\rho^{-1}\quad\text{as }\epsilon\to\rho.
\end{equation}
From \eqref{Lepsilon} and \eqref{Lepsiloncont} we then get
\begin{equation*}
    \frac{\mathcal{L}_\epsilon^{-1}-\mathcal{L}_\rho^{-1}}{\epsilon-\rho}\to -\mathcal{L}_\rho^{-1}L\mathcal{L}_\rho^{-1}\quad\text{as }\epsilon\to\rho,
\end{equation*}
which clearly achieves the proof.
\end{proof}
Following Lemma \ref{lem3} we can rewrite \eqref{reynoldsgen0} as 
\begin{equation}\label{reynoldsgen}
    \partial_x(A(\rho)\partial_x\rho)+\partial_y(A(\rho)\partial_y\rho)=0\qquad\text{in }\omega.
\end{equation}
The equation \eqref{reynoldsgen} is \emph{the generalized Reynolds equation}. From \eqref{bc2} we can write
\begin{equation}\label{reynoldsbc}
    \rho=\rho_0\quad\text{on }\partial\omega.
\end{equation}
We can now prove the main result of this section:
\begin{theorem}\label{thm2}
Recall that we have assumed that $\omega$ is of class $C^\infty$, and that $\rho_0\in C^\infty(\partial\omega)$ and $\rho_0>0$ on $\partial\omega$. Then the problem \eqref{reynoldsgen}--\eqref{reynoldsbc} has a unique solution $\rho\in C^\infty(\ol{\omega})$ satisfying $\rho>0$ on $\omega$.
\end{theorem}
\begin{proof}
Let $\rho_m>0$ be such that
\begin{equation*}
    \rho_m<\inf_{\partial\omega}\rho_0.
\end{equation*}
Define $G:(\rho_m,\infty)\to(0,\infty)$ by
\begin{equation*}
    G(\rho)=\int_{\rho_m}^\rho A(\rho)d\alpha.
\end{equation*}
Since $A$ is strictly positive and continuous we have that $G$ is strictly increasing and differentiable with $G'=A$.\\
Let us formally define 
\begin{equation*}
    \gamma(x,y)=\int_{\rho_m}^{\rho(x,y)} G(\alpha)d\alpha.
\end{equation*}
Then clearly
\begin{equation}\label{eqgamma}
    \nabla\gamma=A(\rho)\nabla\rho,
\end{equation}
and the system \eqref{reynoldsgen}--\eqref{reynoldsbc} can be written as
\begin{equation}\label{eqlaplace}
\left\{
    \begin{aligned}
    \Delta\gamma&=0 &\text{on }\omega,\\
    \gamma&=\int_{\rho_m}^{\rho_0(x,y)} G(\alpha)d\alpha &\text{on }\partial\omega.
    \end{aligned}
    \right.
\end{equation}
 The regularity assumptions, as well as Lemma \ref{lem3} imply that $$(x,y)\to\int_{\rho_m}^{\rho_0(x,y)} G(\alpha)d\alpha\in C^\infty(\partial\omega),$$ 
 and so standard elliptic theory ensures the existence of a solution $\gamma\in C^\infty(\ol{\omega})$ to problem \eqref{eqlaplace}. Moreover, the maximum principle gives us that 
 \begin{equation}\label{maxgamma}
\gamma(x,y)\leq\int_{\rho_m}^{{\rho_M}} G(\alpha)d\alpha\quad\forall\,(x,y)\in\ol{\omega},
 \end{equation}
 where $\rho_M=\sup_{\partial\omega}\rho_0$. Define $J:(\rho_m,\infty)\to(0,\infty)$
 \begin{equation*}
     J(\tau)=\int_{\rho_m}^{\tau} G(\alpha)d\alpha.
 \end{equation*}
 Since $G>0$ we have that $J$ is strictly increasing and from \eqref{maxgamma} we deduce that 
 $$ \gamma(x,y)\in R(J)\quad\forall\,(x,y)\in\ol{\omega}. $$
 Hence we can properly define 
 \begin{equation}\label{properrho}
     \rho(x,y)=J^{-1}\left( \gamma(x,y)\right)\quad\forall\,(x,y)\in\ol{\omega}.
 \end{equation}
 Since $J\in C^\infty$ and $J'=G>0$, it follows that $J^{-1}\in C^\infty$ which implies, through \eqref{properrho}, that $\rho\in C^\infty(\ol{\omega}).$ Evidently, \eqref{eqgamma} holds true, and since $\gamma$ solves the Laplace equation, it follows immediately that $\rho$ solves \eqref{reynoldsgen} in the classical sense. Note that, by definition $\rho$ satisfies \eqref{reynoldsbc} and $\rho\geq \rho_m$ in $\omega$. Lastly, observe that, by construction $\rho$ is unique.
\end{proof}

\section{Estimates on the remainder}

With $\rho$ (and hence $\mathcal{L}_\rho^{-1}$) properly determined we go back to \eqref{solg1} and take $a^1\equiv 0$ in order to fix $g^1$. We can now finally give the term $\ol{g^1}$ appearing in \eqref{bci2} as 
\begin{equation}\label{bcf1}
    \ol{g^1}=g^1_{\,{\big|{(\gamma_l)}_{-}}}\quad\text{which implies }f^1_{\,{\big|{(\gamma_l)}_{-}}}=M\ol{g^1}.
\end{equation}
Next, by taking $f^2=Mg^2$ in \eqref{ord2} we can rewrite it as:
\begin{equation*}
    \mathcal{L}_\rho g^2=2\Gamma(g^1,g^1)-v_x\partial_x g^1-v_y\partial_y g^1. 
\end{equation*}
From \eqref{bc1} it follows that $g^2$ has the diffusive reflection boundary conditions \eqref{bcg1}. From Theorem \ref{thm1} we get that
\begin{equation}\label{solg2}
     g^2=\mathcal{L}_\rho^{-1}\left(2\Gamma(g^1,g^1)-v_x\partial_x g^1-v_y\partial_y g^1\right)+a^2(x,y).
\end{equation}
Once again, we choose $a^2\equiv 0$ in order to fix $g^2$. From\eqref{solg1}, observe that $g^1$ depends on $x$ and $y$ only through $\rho$ (and $\mathcal{L}_\rho^{-1}$). Since $\rho\in C^\infty(\ol{\omega})$ by Theorem \ref{thm2} and the mapping $\rho\to\mathcal{L}_\rho^{-1}$ is $C^\infty$ by Lemma \ref{lem3} we find that
\begin{equation}\label{g1prop}
    (x,y)\to g^1(x,y,z,v)\in C^\infty(\ol{\omega}).
\end{equation}
From \eqref{solg2} and \eqref{g1prop} it follows that we have also
\begin{equation}\label{g2prop}
    (x,y)\to g^2(x,y,z,v)\in C^\infty(\ol{\omega}).
\end{equation}
Note that, by definition, $g^2\in L^2(0,H;L^2(Mdv))$, which, combined with \eqref{g2prop} yields 
\begin{equation}\label{eqt}
    -v_x\partial_x g^2-v_y\partial_y g^2:=t\in L^2(D;L^2(Mdv)).
\end{equation}
As it turns out, it is more convenient to return to the thin domain $D_\varepsilon$ -- we will consider the following truncated expansion of $f_\varepsilon$:
\begin{equation*}
    f_\varepsilon=\rho M+\varepsilon f_\varepsilon^1+\varepsilon^2 f_\varepsilon^2+\varepsilon^{3/2} R_\varepsilon.
\end{equation*}
By factoring $M$ we find:
\begin{equation}\label{exptrunc}
    M^{-1}f_\varepsilon=\rho M+\varepsilon g_\varepsilon^1+\varepsilon^2 g_\varepsilon^2+\varepsilon^{3/2} r_\varepsilon,
\end{equation}
with $r_\varepsilon=M^{-1}R_\varepsilon$. From \eqref{bltze}, as well as \eqref{solg1} and \eqref{solg2}, we have that $r_\varepsilon$ satisfies the following equation:
\begin{equation}\label{remainder}
    v\cdot\nabla_{\bold{x}}r_\varepsilon+\frac{\rho}{\varepsilon}Lr_\varepsilon=s_\varepsilon(r_\varepsilon)+\underbrace{\varepsilon^{1/2} t_\varepsilon+\varepsilon^{1/2}\Gamma\left(g_\varepsilon^1,g_\varepsilon^2\right)+2\varepsilon^{3/2}\Gamma\left(g_\varepsilon^2,g_\varepsilon^2\right)}_{:=w_\varepsilon},
\end{equation}
where 
\begin{equation}\label{eqs}
    s_\varepsilon(r_\varepsilon)=\Gamma\left(r_\varepsilon,g_\varepsilon^1\right)+\varepsilon\Gamma\left(r_\varepsilon,g_\varepsilon^2\right)+2\varepsilon^{1/2}\Gamma\left(r_\varepsilon,r_\varepsilon\right).
\end{equation}
Clearly, $r_\varepsilon$ solves the diffusive reflection boundary conditions on $\omega'$, namely
\begin{equation}\label{bcr1}
\left\{
\begin{aligned}
      r_\varepsilon(x',0,v)_{\big|v_z>0}&=\beta_{r_\varepsilon}(x',0),\\
      r_\varepsilon(x',\varepsilon H,v)_{\big|v_z<0}&=\beta_{r_\varepsilon}(x',\varepsilon H),
\end{aligned}
\right.
\end{equation}
for all $x'=(x,y)\in \omega.$ Moreover, from \eqref{bci2}, \eqref{bcf1}, as well as \eqref{remainder} we get that
\begin{equation}\label{bcr2}
    r_\varepsilon(\bold{x},v)_{\big|v\cdot n_{\bold{x}}<0}=-\varepsilon^{1/2}\ol{g^2_\varepsilon}(\bold{x},v),\quad \text{for }\bold{x}\in\gamma_l^\varepsilon,
\end{equation}
where $\ol{g^2_\varepsilon}$ is the restriction of $g^2_\varepsilon$ to $(\gamma_l^\varepsilon)_-$. We can now state the main result of this section.
\begin{theorem}\label{thm3}
Provided $\varepsilon>0$ is small enough, the problem \eqref{remainder} with boundary conditions \eqref{bcr1}--\eqref{bcr2} has a unique solution $r_\varepsilon\in L^2(D_\varepsilon;L^2(Mdv))$. In particular, this implies that the problem \eqref{bltze} with boundary conditions \eqref{bci1}--\eqref{bci2} has a unique solution $f_\varepsilon$ which satisfies
$$\|f_\varepsilon-\rho M\|\lesssim\varepsilon.$$
\end{theorem}
\subsection{Study of the linear problem}
Let us begin by making a couple of observations. Firstly, for any $h\in L^2(D;L^2(Mdv))$, by a simple scaling $\|h_\varepsilon\|=\varepsilon^{1/2}\|h\|$. Moreover, it is a well known result (for instance from \cite{cercignani1994}) that 
\begin{equation*}
    \left| \Gamma(h_1,h_2) \right|\lesssim \|h_1\|_\nu \|h_2\|_\nu, \quad\forall\,h_1,h_2\in L^2(D;L^2(\nu^{1/2}Mdv)).
\end{equation*}
From \eqref{remainder} we can then write
\begin{equation}\label{estw}
    \|w_\varepsilon\|\lesssim \varepsilon, \quad\text{for }\varepsilon<1.
\end{equation}
Once again, from a scaling argument we get $|\ol{g^2_\varepsilon}|_{\gamma_l^\varepsilon,-}=\varepsilon^{1/2}|\ol{g^2}|_{\gamma_l,-}$, and so from \eqref{bcr2} we obtain
\begin{equation}\label{rgamma}
    \left| r_\varepsilon \right|_{\gamma_l^\varepsilon,-}\lesssim\varepsilon.
\end{equation}
Lastly, let us notice that $s_\varepsilon(r_\varepsilon)\in \left(\ker L\right)^\perp$.
\\
We proceed now to the study of the linear equation
\begin{equation}\label{remainderlin}
    v\cdot\nabla_{\bold{x}}r_\varepsilon+\frac{\rho}{\varepsilon}Lr_\varepsilon=s_\varepsilon+w_\varepsilon
\end{equation}
with fixed $s_\varepsilon\in L^2(D_\varepsilon;L^2(Mdv))\cap \left(\ker L\right)^\perp$, with boundary conditions \eqref{bcr1}--\eqref{bcr2}. We're going to proceed in four steps, just like we have done in the previous section. While Step 1, 2 and 4 are very similar to those for the one dimensional problem and we will skip them, the estimates in Step 3 will be more subtle to obtain. While one difficulty is clearly due to the higher dimension, another one is related to the fact that the size of the domain (and so various constants which are determined by it) depend on $\varepsilon$. \\
Applying Green's theorem in \eqref{remainderlin} and  using \eqref{step3-2} we find 
\begin{equation}\label{part3-1}
    \left|r_\varepsilon-\beta_{r_{\varepsilon}}\right|^2_{\omega',+}+\left|r_\varepsilon \right|^2_{\gamma_l^\varepsilon,+}+\frac{1}{\varepsilon}\int_{D_\varepsilon}\rho (Lr_\varepsilon,r_\varepsilon)d\bold{x}=\int_{D_\varepsilon}(s_\varepsilon+w_\varepsilon,r_\varepsilon)d\bold{x}+\left|r_\varepsilon \right|^2_{\gamma_l^\varepsilon,-}.
\end{equation}
Since $s_\varepsilon\in\left(\ker L\right)^\perp$, we have the estimate:
\begin{equation}\label{part3-2}
    \int_{D_\varepsilon}(s_\varepsilon+w_\varepsilon,r_\varepsilon)d\bold{x}\lesssim \|s_\varepsilon\| \|r_\varepsilon^\perp\|+\|w_\varepsilon\|\|r_\varepsilon\|.
\end{equation}
Let us call
\begin{equation}\label{part3-3}
    \frac{1}{\varepsilon}\left\|r_\varepsilon^\perp\right\| =A_\varepsilon,\quad \left\|P r_\varepsilon\right\| =B_\varepsilon,\quad \left|r_\varepsilon-\beta_{r_{\varepsilon}}\right|_{\omega',+}=C_\varepsilon,\quad\left|r_\varepsilon \right|_{\gamma_l^\varepsilon,+}=\Sigma_\varepsilon.
\end{equation}
By using \eqref{estw}, \eqref{part3-2}, \eqref{part3-3}, \eqref{rgamma} and \eqref{coer} in \eqref{part3-1} we find that
\begin{equation}\label{part3kinest}
    \frac{1}{\varepsilon}\left( C_\varepsilon^2+\Sigma_\varepsilon^2\right)+ A_\varepsilon^2\lesssim \|s_\varepsilon\| A_\varepsilon+ A_\varepsilon+B_\varepsilon+\varepsilon,
\end{equation}
where we have also tacitly used the fact that $\rho\geq \rho_m>0$ on $\ol{\omega}$. \\
The more difficult step is to obtain estimates on the fluid part of $r_\varepsilon$, that is
\begin{equation*}
    Pr_\varepsilon=a_\varepsilon+b_\varepsilon\cdot v+c_\varepsilon\frac{|v|^2-3}{2}.
\end{equation*}
A useful result in obtaining the estimates for $Pr_\varepsilon$ is the following Lemma, which shows how various constants (coming from Poincar\'e, trace or regularity inequalities) depend with respect to $\varepsilon$.
\begin{lemma}\label{lem4}
Let $\varphi_\varepsilon\in H^2(D_\varepsilon)$ satisfying one of the following two boundary conditions:
\begin{itemize}
    \item[(i)] $\varphi_\varepsilon=0$ on $\partial D_\varepsilon$.
    \vspace{-0.4cm}\\
    \item[(ii)] $\partial_z\varphi_\varepsilon=0$ on $\omega'$ and $\varphi_\varepsilon=0$ on $\gamma_l^\varepsilon$.
\end{itemize}
Then we have:
\begin{align}
    \|\varphi_\varepsilon\|_2+\|\nabla\varphi_\varepsilon\|_2 &\lesssim \|D^2\varphi_\varepsilon\|_2,\label{lem41}\\
    |\nabla\varphi_\varepsilon|_{L^2\left(\gamma_l^\varepsilon\right)}&\lesssim \|\varphi_\varepsilon\|_{H^2(D_\varepsilon)},\label{lem42}\\
    |\nabla\varphi_\varepsilon|_{L^2\left(\omega'\right)}&\lesssim \varepsilon^{-1/2}\|\varphi_\varepsilon\|_{H^2(D_\varepsilon)},\label{lem44}\\
    \|\varphi_\varepsilon\|_{H^2(D_\varepsilon)}&\lesssim \|\Delta \varphi_\varepsilon\|_2.\label{lem43}
\end{align}
\end{lemma}
\begin{proof}
We will prove the result for $\varphi_\varepsilon\in C^3(\ol{D_\varepsilon})$, and the result will follow by density.\\
Since $\varphi_\varepsilon=0$ on $\gamma_l^\varepsilon$ in both cases, by Poincar\'e's inequality we can write, for instance
\begin{equation}\label{proof41}
    \|\varphi_\varepsilon\|_2\lesssim \|\partial_x \varphi_\varepsilon\|_2,\quad  \|\varphi_\varepsilon\|_2\lesssim \|\partial_y \varphi_\varepsilon\|_2,
\end{equation}
since the size of $\omega$ is of order $1$. Similarly, since $\partial_z\varphi_\varepsilon=0$ on $\gamma_l^\varepsilon$, we obtain
\begin{equation}\label{proof42}
    \|\partial_z\varphi_\varepsilon\|_2\lesssim \|\partial_{zx} \varphi_\varepsilon\|_2.
\end{equation}
An easy integration by parts gives us:
\begin{equation*}
    \int_{D_\varepsilon}\varphi_\varepsilon\partial_{xx}\varphi_\varepsilon d\bold{x}=-\int_{D_\varepsilon}\left(\partial_{x}\varphi_\varepsilon\right)^2 d\bold{x}+\int_{\partial D_\varepsilon}\varphi_\varepsilon \partial_{x}\varphi_\varepsilon n_x dS.
\end{equation*}
The boundary term in the above relation vanishes since, on $\omega'$, $n_x=0$, while $\varphi_\varepsilon=0$ on $\gamma_l^\varepsilon$. Hence, with Cauchy-Schwarz inequality we obtain
\begin{equation*}
    \|\partial_x\varphi_\varepsilon\|_2^2\leq \| \varphi_\varepsilon\|_2\|\partial_{xx} \varphi_\varepsilon\|_2,
\end{equation*}
and so by using \eqref{proof41} we get:
\begin{equation}\label{proof43}
    \|\partial_x\varphi_\varepsilon\|_2\lesssim \|\partial_{xx} \varphi_\varepsilon\|_2.
\end{equation}
In a similar fashion we can prove that:
\begin{equation}\label{proof44}
    \|\partial_y\varphi_\varepsilon\|_2\lesssim \|\partial_{yy} \varphi_\varepsilon\|_2.
\end{equation}
From \eqref{proof41}--\eqref{proof44} the estimate \eqref{lem41} clearly follows.\\
Since $\partial_z\varphi_\varepsilon=0$ on $\gamma_l^\varepsilon$, we only need to prove \eqref{lem42} for $\partial_x\varphi_\varepsilon$ (the $y$ derivative is completely similar). We will do that by a simple scaling. The trace inequality in the fixed domain $D$ gives us:
\begin{equation*}
    |\partial_x\varphi|_{\gamma_l}\lesssim \|\partial_x\varphi\|_2+\|\partial_{xx}\varphi\|_2+\|\partial_{xz}\varphi\|_2+\|\partial_{zz}\varphi\|_2+\text{similar terms}.
\end{equation*}
Note that
\begin{align*}
    |\partial_x\varphi_\varepsilon|_{L^2(\gamma_l^\varepsilon)}=&\varepsilon^{1/2}|\partial_x\varphi|_{L^2(\gamma_l)},\quad \|\partial_x\varphi_\varepsilon\|_2=\varepsilon^{1/2}\|\partial_x\varphi\|_2,\quad \|\partial_{xx}\varphi_\varepsilon\|_2=\varepsilon^{1/2}\|\partial_{xx}\varphi\|_2,\\
    &\|\partial_{xz}\varphi_\varepsilon\|_2=\varepsilon^{-1/2}\|\partial_{xz}\varphi\|_2,\quad \|\partial_{zz}\varphi_\varepsilon\|_2=\varepsilon^{-3/2}\|\partial_{zz}\varphi\|_2.
\end{align*}
This clearly implies that, the optimal trace inequality constant actually decreases with $\varepsilon$, justifying \eqref{lem42}. The proof of \eqref{lem44} is done exactly in the same way, the worse constant on the right hand side coming from the fact 
$$ |\partial_x\varphi_\varepsilon|_{L^2(\omega')}=|\partial_x\varphi|_{L^2(\omega')},$$
while the other scaling constants remain the same.
\\
The trickiest thing to prove is \eqref{lem43}. First, notice that, by \eqref{lem41} it is enough to obtain the bound on $D^2\varphi_\varepsilon$. We can write
    \begin{align}
          \|\Delta\varphi_\varepsilon\|^2&=\int_{D_\varepsilon}\left(\partial_{xx}\varphi_\varepsilon+\partial_{yy}\varphi_\varepsilon+\partial_{zz}\varphi_\varepsilon\right)^2d\bold{x}=\sum_x \int_{D_\varepsilon}\left(\partial_{xx}\varphi_\varepsilon\right)^2d\bold{x}\notag\\
          &+2\int_{D_\varepsilon}\left(\partial_{xx}\varphi_\varepsilon\partial_{yy}\varphi_\varepsilon+\partial_{xx}\varphi_\varepsilon\partial_{zz}\varphi_\varepsilon+\partial_{yy}\varphi_\varepsilon\partial_{zz}\varphi_\varepsilon\right)d\bold{x}.\label{proof45}
    \end{align}
    Integrating twice by parts gives us:
    \begin{equation*}
        \int_{D_\varepsilon}\partial_{xx}\varphi_\varepsilon\partial_{zz}\varphi_\varepsilon d\bold{x}=\int_{D_\varepsilon}\left(\partial_{xz}\varphi_\varepsilon \right)^2 d\bold{x}+\underbrace{\int_{\partial D_\varepsilon}\partial_{x}\varphi_\varepsilon\partial_{zz}\varphi_\varepsilon n_x dS}_{T_1}-\underbrace{\int_{\partial D_\varepsilon}\partial_{x}\varphi_\varepsilon\partial_{xz}\varphi_\varepsilon n_z dS}_{T_2}.
    \end{equation*}
    Clearly $T_1=0$ on $\omega'$ since $n_x=0$ and $T_2=0$ on $\gamma_l^\varepsilon$, since $n_z=0$. Moreover, as we have argued before $\partial_{zz}\varphi_\varepsilon=0$ on $\gamma_l^\varepsilon$. Hence
    \begin{equation*}
        T_1-T_2=\int_{\omega'}\partial_{x}\varphi_\varepsilon\partial_{xz}\varphi_\varepsilon n_z dS.
    \end{equation*}
    If $\varphi_\varepsilon=0$ on $\omega'$ then clearly $\partial_{x}\varphi_\varepsilon=0$. If, on the other hand,  $\partial_z \varphi_\varepsilon=0$ on $\omega'$ then $\partial_{zx}\varphi_\varepsilon=0$. Either way $T_1=T_2=0$, showing
    \begin{equation}\label{proof46}
        \int_{D_\varepsilon}\partial_{xx}\varphi_\varepsilon\partial_{zz}\varphi_\varepsilon d\bold{x}=\|\partial_{xz}\varphi_\varepsilon\|^2_2.
    \end{equation}
    In an identical manner we can prove:
    \begin{equation}\label{proof47}
        \int_{D_\varepsilon}\partial_{yy}\varphi_\varepsilon\partial_{zz}\varphi_\varepsilon d\bold{x}=\|\partial_{yz}\varphi_\varepsilon\|^2_2.
    \end{equation}
Lastly, we have
\begin{align}
        2\int_{D_\varepsilon}\partial_{xx}\varphi_\varepsilon\partial_{yy}\varphi_\varepsilon d\bold{x}=&2\int_{D_\varepsilon}\left(\partial_{xy}\varphi_\varepsilon \right)^2 d\bold{x}+\int_{\gamma_l^\varepsilon}\underbrace{\left(\partial_{x}\varphi_\varepsilon\partial_{yy}\varphi_\varepsilon n_x-\partial_{x}\varphi_\varepsilon\partial_{xy}\varphi_\varepsilon n_y\right)}_{T_3} dS\notag\\
        +&\int_{\gamma_l^\varepsilon}\underbrace{\left(\partial_{y}\varphi_\varepsilon\partial_{xx}\varphi_\varepsilon n_y-\partial_{y}\varphi_\varepsilon\partial_{xy}\varphi_\varepsilon n_x\right)}_{T_4} dS.\label{proof48}
    \end{align}
Recalling the definition of $\partial\omega$ we get that the outward normal to $\gamma_l^\varepsilon$ is given by 
\begin{equation}\label{proof411}
    n=\bigg(\underbrace{\frac{y_0'}{\sqrt{(y_0')^2+1}}}_{n_x},\underbrace{-\frac{1}{\sqrt{(y_0')^2+1}}}_{n_y},0\bigg).
\end{equation}
The boundary condition of $\varphi_\varepsilon$ on $\gamma_l^\varepsilon$ implies that
\begin{align}
    0&=\frac{d}{dx} \varphi_\varepsilon(x,y_0(x),z)=\partial_x\varphi_\varepsilon+\partial_y\varphi_\varepsilon y_0',\label{proof49}\\
    0&=\frac{d^2}{dx^2} \varphi_\varepsilon(x,y_0(x),z)=\partial_{xx}\varphi_\varepsilon+\partial_{xy}\varphi_\varepsilon y_0'+y_0'\left( \partial_{xy}\varphi_\varepsilon +\partial_{yy}\varphi_\varepsilon y_0'\right)+\partial_y\varphi_\varepsilon y_0''.\label{proof410}
\end{align}
By multiplying in \eqref{proof410} with $\partial_y\varphi_\varepsilon$ we can write, using \eqref{proof49} that
\begin{equation*}
    0=-T_4\sqrt{(y_0')^2+1}-T_3\sqrt{(y_0')^2+1}+y_0''\left(\partial_y\varphi_\varepsilon\right)^2.
\end{equation*}
Going back to \eqref{proof48} we find that
\begin{align}
    \int_{D_\varepsilon}\partial_{xx}\varphi_\varepsilon\partial_{yy}\varphi_\varepsilon d\bold{x}&=
   \|\partial_{xy}\varphi_\varepsilon \|^2+\frac{1}{2}\int_{\gamma_l^\varepsilon}\frac{y_0''}{\sqrt{(y_0')^2+1}}(\partial_{y}\varphi_\varepsilon)^2 dS\notag\\
   &\leq \|\partial_{xy}\varphi_\varepsilon \|^2+\frac{1}{2}\sup |y_0''| \left|\partial_{y}\varphi_\varepsilon\right|_{\gamma_l^\varepsilon}^2\lesssim \|D^2 \varphi_\varepsilon\|_2,\label{proof412} 
\end{align}
  by using \eqref{lem42} and \eqref{lem41}.\\
The desired result follows from \eqref{proof412}, \eqref{proof47}, \eqref{proof46} and \eqref{proof45}.\\
\end{proof}
We are now ready to pass to the derivation of the estimates on the fluid part of $r_\varepsilon$. Start with the following weak formulation of problem \eqref{remainderlin}:
\begin{align}\label{weakkin}
        \int_{D_\varepsilon\times\mathbb{R}^3}r_\varepsilon v\cdot\nabla_{\bold{x}}\psi_\varepsilon +\int_{\omega'\times\mathbb{R}^3}\beta_{r_{\varepsilon}}\psi_\varepsilon v\cdot n_{\bold{x}}=\int_{\omega'\times\mathbb{R}^3}\left(r_\varepsilon-\beta_{r_{\varepsilon}}\right)\psi_\varepsilon v\cdot n_{\bold{x}}\notag\\
        +\frac{1}{\varepsilon}\int_{D_\varepsilon\times\mathbb{R}^3} \rho Lr_\varepsilon^\perp\psi_\varepsilon
    -\int_{D_\varepsilon\times\mathbb{R}^3}(s_\varepsilon+w_\varepsilon)\psi_\varepsilon+\int_{\gamma_l^\varepsilon\times\mathbb{R}^3} r_\varepsilon\psi_\varepsilon v\cdot n_{\bold{x}},
\end{align}
where the tacit measures are $Md\bold{x}dv$ for the bulk terms and $MdS_{\bold{x}}dv$ for the boundary ones.
In order  the key idea is -- as first shown in \cite{esposito2013} -- to pick appropriate test functions $\psi_\varepsilon$ in \eqref{weakkin}.\\
\vspace{.1cm}\\
\underline{Estimate on $c_\varepsilon$:} Choose first the test function
\begin{equation*}
    \psi=\psi_{c_{\varepsilon}}=(|v|^2-5)v\cdot\nabla_{\bold{x}}\phi_{c_\varepsilon},
\end{equation*}
where $\phi_{c_\varepsilon}=\phi_{c_\varepsilon}(\bold{x})$ solves the elliptic problem
\begin{equation}\label{cepsilon}
\left\{
    \begin{aligned}
    -\Delta\phi_{c_\varepsilon}&=c_\varepsilon &\text{on }D_\varepsilon,\\
    \phi_{c_\varepsilon}&=0 &\text{on }\partial D_\varepsilon.
    \end{aligned}
    \right.
\end{equation}
The right hand side of \eqref{weakkin}, which we will call $(*)_{c_\varepsilon}$ is controlled as follows
\begin{equation*}
\begin{split}
      \left|(*)_{c_\varepsilon}\right|\lesssim \left|r_\varepsilon-\beta_{r_\varepsilon}\right|_{\omega',+}\left|\nabla\phi_{c_\varepsilon}\right|_{L^2(\omega')}+\frac{1}{\varepsilon}\|r_\varepsilon^\perp\| \|\nabla\phi_{c_\varepsilon}\|_2+\left(\|s_\varepsilon\|+\|w_\varepsilon\|\right)\|\nabla\phi_{c_\varepsilon}\|_2\\
      +\left(|r_\varepsilon|_{\gamma_l^\varepsilon,-}
      +|r_\varepsilon|_{\gamma_l^\varepsilon,+}\right)\left|\nabla\phi_{c_\varepsilon}\right|_{L^2(\gamma_l^\varepsilon)}.
\end{split}
\end{equation*}
Taking \eqref{cepsilon} into account, using the results of Lemma \ref{lem4} and recalling \eqref{estw}, \eqref{rgamma} and the notations \eqref{part3-3} we find
\begin{equation}\label{ceps0}
    \left|(*)_{c_\varepsilon}\right|\lesssim \left(\varepsilon^{-1/2}C_\varepsilon+A_\varepsilon+\|s_\varepsilon\|+\varepsilon+\Sigma_\varepsilon\right)\|c_\varepsilon\|_2.
\end{equation}
Following \cite{esposito2013} we can show that
\begin{align}
    \int_{\omega'\times\mathbb{R}^3}\beta_{r_{\varepsilon}}\psi_{c_\varepsilon} v\cdot n_{\bold{x}}=0,\label{ceps1}\\
    \int_{D_\varepsilon\times\mathbb{R}^3}r_\varepsilon v\cdot\nabla_{\bold{x}}\psi_{c_\varepsilon}=-10\pi\|c_\varepsilon\|^2_2+\int_{D_\varepsilon\times\mathbb{R}^3}r_\varepsilon^\perp (|v|^2-5) v_i v_j\partial_{ij}\phi_{c_\varepsilon},\label{ceps2}
\end{align}
using Einstein's summation convention for $i,j\in\{1,2,3\}$ and $\partial_1=\partial_x,\,\partial_2=\partial_y,\,\partial_3=\partial_z$. Using Lemma \ref{lem4} and \eqref{cepsilon} we readily find
\begin{equation}\label{ceps3}
    \left|\int_{D_\varepsilon\times\mathbb{R}^3}r_\varepsilon^\perp (|v|^2-5) v_i v_j\partial_{ij}\phi_{c_\varepsilon}\right|\lesssim \|r_\varepsilon^\perp\|\|c_\varepsilon\|_2.
\end{equation}
From \eqref{weakkin}, \eqref{ceps0}, \eqref{ceps1}, \eqref{ceps2} and \eqref{ceps3} we obtain
\begin{equation}\label{estcepsilon}
    \|c_\varepsilon\|_2\lesssim\varepsilon^{-1/2}C_\varepsilon+A_\varepsilon+\|s_\varepsilon\|+\varepsilon+\Sigma_\varepsilon.
\end{equation}
\vspace{.1cm}\\
\underline{Estimate on $b_\varepsilon$:} Choose the test functions in \eqref{weakkin}
\begin{equation*}
    \psi_\varepsilon=\psi^{i,j}_{b_{\varepsilon}}=(v_i^2-1)\partial_j\phi^j_{b_\varepsilon}, \quad\text{for }i,j\in\{1,2,3\},
\end{equation*}
where $\phi^j_{b_\varepsilon}=\phi^j_{b_\varepsilon}(\bold{x})$ solve the elliptic problems
\begin{equation}\label{bepsilon1}
\left\{
    \begin{aligned}
    -\Delta\phi^j_{b_\varepsilon}&=b^j_\varepsilon &\text{on }D_\varepsilon,\\
    \phi^j_{b_\varepsilon}&=0 &\text{on }\partial D_\varepsilon.
    \end{aligned}
    \right.
\end{equation}
Similarly to \eqref{ceps0} we have
\begin{equation*}
    \left|(*)_{b_\varepsilon}^{i,j}\right|\lesssim \left(\varepsilon^{-1/2}C_\varepsilon+A_\varepsilon+\|s_\varepsilon\|+\varepsilon+\Sigma_\varepsilon\right)\|b^j_\varepsilon\|_2,\quad\text{for }i,j\in\{1,2,3\}.
\end{equation*}
Once again, as in \cite{esposito2013}, we find
\begin{align*}
    \int_{\omega'\times\mathbb{R}^3}\beta_{r_{\varepsilon}}\psi^{i,j}_{b_\varepsilon} v\cdot n_{\bold{x}}=0,\\
    \int_{D_\varepsilon\times\mathbb{R}^3}r_\varepsilon v\cdot\nabla_{\bold{x}}\psi^{i,j}_{b_\varepsilon}=2\int_{D_\varepsilon}b^i_\varepsilon\partial_{ij}\phi^j_{b_\varepsilon}+\int_{D_\varepsilon\times\mathbb{R}^3}r_\varepsilon^\perp (v_i^2-1) v_j\partial_{ij}\phi^j_{b_\varepsilon},
\end{align*}
for all $i,j\in\{1,2,3\}.$ We then derive the following estimate
\begin{equation}\label{estbepsilon1}
    \left|\int_{D_\varepsilon}b^i_\varepsilon\partial_{ij}\phi^j_{b_\varepsilon}\right|\lesssim \left(\varepsilon^{-1/2}C_\varepsilon+A_\varepsilon+\|s_\varepsilon\|+\varepsilon+\Sigma_\varepsilon\right)\|b^j_\varepsilon\|_2,
\end{equation}
for $i,j\in\{1,2,3\}$.\\
Next, take as test functions in \eqref{weakkin}
\begin{equation*}
    \psi_\varepsilon=\ol{\psi}^{i,j}_{b_{\varepsilon}}=|v|^2v_iv_j\partial_j\phi^i_{b_\varepsilon}, \quad\text{for }i\neq j,
\end{equation*}
where $b^i_\varepsilon$ is defined in \eqref{bepsilon1}. The right hand side of \eqref{weakkin} is bounded as follows:
\begin{equation*}
    \left|\ol{(*)}_{b_\varepsilon}^{i,j}\right|\lesssim \left(\varepsilon^{-1/2}C_\varepsilon+A_\varepsilon+\|s_\varepsilon\|+\varepsilon+\Sigma_\varepsilon\right)\|b^i_\varepsilon\|_2,\quad\text{for }i\neq j.
\end{equation*}
As for the terms on the left hand side we have (\cite{esposito2013}):
\begin{align*}
    \int_{\omega'\times\mathbb{R}^3}\beta_{r_{\varepsilon}}\ol{\psi}^{i,j}_{b_\varepsilon} v\cdot n_{\bold{x}}&=0,\\
    \int_{D_\varepsilon\times\mathbb{R}^3}r_\varepsilon v\cdot\nabla_{\bold{x}}\ol{\psi}^{i,j}_{b_\varepsilon}&=7\left(\int_{D_\varepsilon}b^j_\varepsilon\partial_{ij}\phi^i_{b_\varepsilon}+\int_{D_\varepsilon} b^i_\varepsilon\partial_{jj}\phi^i_{b_\varepsilon}\right)\\
    &+\int_{D_\varepsilon\times\mathbb{R}^3}r_\varepsilon^\perp |v|^2 v_i v_j v_k\partial_{jk}\phi^j_{b_\varepsilon},
\end{align*}
for $i\neq j$. We are led to the following estimate
\begin{equation}\label{estbepsilon2}
    \left|\int_{D_\varepsilon} b^i_\varepsilon\partial_{jj}\phi^i_{b_\varepsilon}\right|\lesssim \left|\int_{D_\varepsilon} b^j_\varepsilon\partial_{ij}\phi^i_{b_\varepsilon}\right|+\left(\varepsilon^{-1/2}C_\varepsilon+A_\varepsilon+\|s_\varepsilon\|+\varepsilon+\Sigma_\varepsilon\right)\|b^i_\varepsilon\|_2
\end{equation}
for $i\neq j$. By combining \eqref{estbepsilon1} and \eqref{estbepsilon2} we obtain
\begin{equation}\label{estbepsilon3}
    \left|\int_{D_\varepsilon} b^i_\varepsilon\partial_{jj}\phi^i_{b_\varepsilon}\right|\lesssim\left(\varepsilon^{-1/2}C_\varepsilon+A_\varepsilon+\|s_\varepsilon\|+\varepsilon+\Sigma_\varepsilon\right)\|b_\varepsilon\|_2
\end{equation}
for $i\neq j$.
By summing in \eqref{estbepsilon3} for $j\neq i$ and using \eqref{estbepsilon1} with $j=i$ we find that
\begin{equation*}
    \left|\int_{D_\varepsilon} b^i_\varepsilon\Delta\phi^i_{b_\varepsilon}\right|\lesssim \left(\varepsilon^{-1/2}C_\varepsilon+A_\varepsilon+\|s_\varepsilon\|+\varepsilon+\Sigma_\varepsilon\right)\|b_\varepsilon\|_2
\end{equation*}
for all $i\in\{1,2,3\}$ and so by using \eqref{bepsilon1} we immediately obtain
\begin{equation}\label{estbepsilon}
   \|b_\varepsilon\|_2\lesssim \varepsilon^{-1/2}C_\varepsilon+A_\varepsilon+\|s_\varepsilon\|+\varepsilon+\Sigma_\varepsilon.
\end{equation}
\vspace{.1cm}\\
\underline{Estimate on $a_\varepsilon$:} Take the test function in \eqref{weakkin}
\begin{equation*}
    \psi=\psi_{a_{\varepsilon}}=(|v|^2-10)v\cdot\nabla_{\bold{x}}\phi_{a_\varepsilon},
\end{equation*}
where $\phi_{a_\varepsilon}=\phi_{a_\varepsilon}(\bold{x})$ solves the elliptic problem with mixed boundary conditions
\begin{equation}\label{cepsilon}
\left\{
    \begin{aligned}
    -\Delta\phi_{a_\varepsilon}&=a_\varepsilon &\text{on }D_\varepsilon,\\
    \partial_z\phi_{a_\varepsilon}&=0 &\text{on } \omega',\\
        \phi_{a_\varepsilon}&=0 &\text{on } \gamma_l^\varepsilon.
    \end{aligned}
    \right.
\end{equation}
Once again, the right hand side of \eqref{weakkin}
\begin{equation*}
    \left|(*)_{a_\varepsilon}\right|\lesssim \left(\varepsilon^{-1/2}C_\varepsilon+A_\varepsilon+\|s_\varepsilon\|+\varepsilon+\Sigma_\varepsilon\right)\|a_\varepsilon\|_2.
\end{equation*}
Following \cite{esposito2013} we can write
\begin{align*}
    \int_{\omega'\times\mathbb{R}^3}\beta_{r_{\varepsilon}}\psi_{a_\varepsilon} v\cdot n_{\bold{x}}=0,\\
    \int_{D_\varepsilon\times\mathbb{R}^3}r_\varepsilon v\cdot\nabla_{\bold{x}}\psi_{a_\varepsilon}=\frac{5}{2}\|a_\varepsilon\|^2_2+\int_{D_\varepsilon\times\mathbb{R}^3}r_\varepsilon^\perp (|v|^2-10) v_i v_j\partial_{ij}\phi_{a_\varepsilon}.
\end{align*}
Like with the estimates for $c_\varepsilon$  we obtain
\begin{equation}\label{estaepsilon}
    \|a_\varepsilon\|_2\lesssim\varepsilon^{-1/2}C_\varepsilon+A_\varepsilon+\|s_\varepsilon\|+\varepsilon+\Sigma_\varepsilon.
\end{equation}
We can now combine \eqref{estcepsilon}, \eqref{estbepsilon} and \eqref{estaepsilon} to obtain
\begin{equation*}
    B_\varepsilon^2=\|Pr_\varepsilon\|_\nu^2\lesssim \left(\varepsilon^{-1/2}C_\varepsilon+A_\varepsilon+\|s_\varepsilon\|+\varepsilon+\Sigma_\varepsilon\right)B_\varepsilon,
\end{equation*}
which implies that
\begin{equation}\label{part3fluidest}
    B_\varepsilon^2\lesssim \varepsilon^{-1}C_\varepsilon^2+A_\varepsilon^2+\Sigma_\varepsilon^2+(\|s_\varepsilon\|+\varepsilon) B_\varepsilon.
\end{equation}
By combining \eqref{part3kinest} and \eqref{part3fluidest} we obtain, after a couple of elementary algebraic manipulations that
\begin{equation}\label{lastest}
    A_\varepsilon+B_\varepsilon\lesssim \|s_\varepsilon\|+1.
\end{equation}

\subsection{The nonlinear problem}
 Define the sequence $r^k_\varepsilon$ by 
    \begin{equation*}
        v\cdot\nabla_\bold{x} r^k_\varepsilon+\frac{\rho}{\varepsilon}Lr^k_\varepsilon=s_\varepsilon(r^{k-1}_\varepsilon)+w_\varepsilon,
    \end{equation*}
    with boundary conditions \eqref{bcr1}--\eqref{bcr2} for $k\geq 1$ and $r^0_\varepsilon=0.$ The linear theory in the previous section ensures that $r^k_\varepsilon$ is well defined. The estimate \eqref{lastest} gives us
    \begin{equation}
        \|r^k_\varepsilon\|_\nu\lesssim 1+\|s_\varepsilon(r^{k-1}_\varepsilon)\|.
    \end{equation}
    Since 
    \begin{align*}
        \left\|\Gamma\left( g^i_\varepsilon,r^{k-1}_\varepsilon\right)\right\|\lesssim \|g^i_\varepsilon\|_\nu \|r^{k-1}_\varepsilon\|_\nu&=\varepsilon^{1/2}\|g^i\|_\nu \|r^{k-1}_\varepsilon\|_\nu,\quad\text{for }i\in\{1,2\},\\
        \left\|\Gamma\left( r^{k-1}_\varepsilon,r^{k-1}_\varepsilon\right)\right\|&\lesssim \|r^{k-1}_\varepsilon\|_\nu^2
    \end{align*}
    and so from \eqref{eqs} we find that
    \begin{equation*}
         \|r^k_\varepsilon\|_\nu\lesssim 1+\varepsilon^{1/2} \|r^{k-1}_\varepsilon\|_\nu+\varepsilon^{1/2}\|r^{k-1}_\varepsilon\|_\nu^2.
    \end{equation*}
    The above relation clearly implies that, for $\varepsilon$ small enough
    \begin{equation}\label{estrk}
        \|r^k_\varepsilon\|_\nu\lesssim 1\quad\text{uniformly in } k. 
    \end{equation}
    Set 
$$ q^k_\varepsilon= r^{k}_\varepsilon-r^{k-1}_\varepsilon\quad\text{for } k\geq 2.$$
Then $q^k_\varepsilon$ solves the following
$$ v\cdot\nabla_\bold{x} q^k_\varepsilon+\frac{\rho}{\varepsilon}Lq^k_\varepsilon= j_\varepsilon(q^{k-1}_\varepsilon),$$
for $k\geq 3$ where 
$$j_\varepsilon(q^{k-1}_\varepsilon)=\Gamma\left( g^1_\varepsilon,q^{k-1}_\varepsilon\right)+\varepsilon\Gamma\left( g^2_\varepsilon,q^{k-1}_\varepsilon\right)+2\varepsilon^{1/2}\Gamma\left( r^{k-1}_\varepsilon+r^{k-2}_\varepsilon,q^{k-1}_\varepsilon\right).$$
The boundary conditions for $q^k_\varepsilon$ are obviously \eqref{bcr1} on $\omega'$, while 
\begin{equation}\label{bcq}
    q^k_\varepsilon(\bold{x},v)_{\big|v\cdot n_{\bold{x}}<0}=0,\quad \text{for }\bold{x}\in\gamma_l^\varepsilon.
\end{equation}
Using this boundary condition we can improve the estimates for the linear problem to find
 \begin{equation*}
        \|q^k_\varepsilon\|_\nu\lesssim \|j_\varepsilon(q^{k-1}_\varepsilon)\|, 
    \end{equation*}
    and using \eqref{estrk} we find that 
    \begin{equation*}
        \|q^k_\varepsilon\|_\nu\lesssim \varepsilon^{1/2}\|q^{k-1}_\varepsilon\|,
    \end{equation*}
    which shows that, for $\varepsilon$ small enough, there exists $\zeta\in (0,1)$ with
    \begin{equation}
        \|q^k_\varepsilon\|_\nu \leq \zeta\quad\text{uniformly in } k.
    \end{equation}
    This implies that $r_\varepsilon^k$ is strongly convergent in $L^2(D_\varepsilon;L^2(\nu^{1/2} Mdv))$ to $r_\varepsilon$, which is clearly the solution to the original problem \eqref{remainder}. This clearly achieves the proof of Theorem \ref{thm3}.
\bibliographystyle{plain}
\bibliography{Boltzmann_to_Reynolds}
\end{document}